\documentclass[11pt]{amsart}
\usepackage[margin=1in]{geometry}
\usepackage{amsmath}
\usepackage{amsfonts,amsmath}
\usepackage{paralist}
\usepackage[colorlinks=true]{hyperref}
\hypersetup{urlcolor=blue, citecolor=red}
 \numberwithin{equation}{section}
\begin{document}
	\newcommand{\R}{\mathbb{R}}
	\newcommand{\RN}{\mathbb{R}^N}
	\newcommand{\supe}{\textup{ess sup}}
	\newcommand{\uu}{{\bf u}}
	\newcommand{\w}{{\bf w}}
	\newcommand{\g}{{\bf f}}
	\newcommand{\fn}{\exp}
	\newcommand{\irn}{{\int_{\RN}}}
	\newcommand*{\avint}{\mathop{\ooalign{$\int$\cr$-$}}}
		\newcommand{\qdz}{Q_{\delta_0}(0)}
	\newcommand{\bdz}{B_{\delta_0}(0)}
	\newcommand{\id}{\int_{Q_{\delta_0}(0)}}
	\newcommand{\iqo}{\int_{Q_{1}(0)}}
	\newcommand{\ibo}{\int_{B_{1}(0)}}
	\newcommand{\ibd}{\int_{B_{\delta_0}(0)}}
	\newcommand{\io}{\int_{\Omega}}
	\newcommand{\iot}{\int_{\Omega_{\tau}}}
	\newcommand{\ioT}{\int_{\Omega_{T}}}
	\newcommand{\iq}{\int_{Q_{r}(z)}}
	\newcommand{\iqh}{\int_{Q_{\frac{r}{2}}(z)}}
	\newcommand{\ib}{\int_{B_r(y)}}
	\newcommand{\ibz}{\int_{B_1(0)}}
	\newcommand{\ibh}{\int_{B_{\frac{r}{2}}(y)}}
	\newcommand{\ibk}{\int_{B_{r_k}(y_k)}}
	\newcommand{\iqk}{\int_{Q_{r_k}(z_k)}}
	\newcommand{\mzk}{m_{z_k,r_k}}
	\newcommand{\qyr}{Q_r(y,\tau) }
	\newcommand{\qzr}{Q_r(z) }
	\newcommand{\qyrk}{Q_{r_k}(y_k,\tau_k) }
	\newcommand{\qzrk}{Q_{r_k}(z_k) }
	\newcommand{\er}{E_r(y,\tau) }
	\newcommand{\ezr}{E_r(z) }
	\newcommand{\erk}{E_{r_k}(y_k,\tau_k) }
	\newcommand{\ezk}{E_{r_k}(z_k) }
	\newcommand{\byr}{B_r(y) }
	\newcommand{\byrk}{B_{r_k}(y_k) }
	\newcommand{\aiy}{\avint_{B_r(y)}  }
	\newcommand{\air}{\avint_{B_\rho(y)}  }
	\newcommand{\aizr}{\avint_{Q_\rho(z)}  }
	\newcommand{\aiz}{\avint_{Q_r(z)}  }
	\newcommand{\aiR}{\avint_{B_R(y)}  }
	\newcommand{\aiyk}{\avint_{B_{r_k}(y_k)}  }
	\newcommand{\aizk}{\avint_{Q_{r_k}(z_k)}  }
	\newcommand{\ykr}{y_k+r_ky}
	\newcommand{\tkr}{\tau_k+r_k^2\tau}
	\newcommand{\maxt}{\max_{t\in[\tau-\frac{1}{2}r^2, \tau+\frac{1}{2}r^2]}}
	\newcommand{\mat}{\max_{t\in[\tau-\frac{1}{2}r^2, \tau+\frac{1}{2}r^2]}}
	\newcommand{\ftq}{\|f\|_{q,\ot}}
	\newcommand{\maxd}{\max_{\tau\in[-\frac{1}{2}\delta^2, \frac{1}{2}\delta^2]}}
	\newcommand{\maxn}{\max_{\tau\in[-\frac{1}{2}, \frac{1}{2}]}}
	\newcommand{\maxtk}{\max_{t\in[\tau_k-\frac{1}{2}r_k^2, \tau_k+\frac{1}{2}r_k^2]}}
	\newcommand{\maxth}{\max_{t\in[\tau-\frac{1}{8}r^2, \tau+\frac{1}{8}r^2]}}
	\newcommand{\mnp}{(\m\cdot\nabla p) }
	\newcommand{\nsk}{(n_k\cdot\nabla \sk) }
	\newcommand{\ot}{\Omega_T }
	\newcommand{\mrt}{\m_{y,r}(t) }
	\newcommand{\mz}{\m_{z,r} }
	\newcommand{\prt}{p_{y,r}(t) }
	\newcommand{\pkt}{p_{y_k,r_k}(\tkr) }
	\newcommand{\wks}{|w_{k}|^2 }
	\newcommand{\sk}{\psi_{k} }
	\newtheorem{theorem}{Theorem}[section]
	\newtheorem{corollary}{theorem}[section]
	\newtheorem*{main}{Main Theorem}
	\newtheorem{lemma}[theorem]{Lemma}
	\newtheorem{clm}[theorem]{Claim}
	\newtheorem{proposition}{Proposition}[section]
	\newtheorem{conjecture}{Conjecture}
	\newtheorem*{problem}{Problem}
	\theoremstyle{definition}
	\newtheorem{definition}[theorem]{Definition}
	\newtheorem{remark}{Remark}[section]
	\newtheorem*{notation}{Notation}
	\newcommand{\rn}{\mathbb{R}^N}
	\newcommand{\gr}{G_\rho}
	\newcommand{\pr}{\phi_\rho}
	\newcommand{\br}{B_{\rho}}
	\newcommand{\mr}{\m_{\rho}}
	\newcommand{\thr}{\theta_{\rho}}
	\newcommand{\f}{{\bf g}}
	\newcommand{\aoo}{1+m_1^2(x_0,t_0)}
	\newcommand{\eps}[1]{{#1}_{\varepsilon}}

	%% Place the running title of the paper with 40 letters or less in []
	%% and the full title of the paper in { }.
	\title[A class of parabolic equations
	] %Use the shortened version of the full title
	{Logarithmic upper bounds for weak solutions to a class of parabolic equations}
	%Existence of classical solutions to a PDE system with cubic nonlinearity
	% Place all authors' names in [ ] shown as running head, Leave { } empty
	% Please use `and' to connect the last two names if applicable
	% Use FirstNameInitial.  MiddleNameInitial. LastName, or only last names of authors if there are too many authors
	\author[Xiangsheng Xu]{}
	
	% It is required to enter 2010 MSC.
	% Please provide minimum  5 keywords.
	\keywords{  Logarithmic upper bounds, Moser's iteration technique, uniform bounds for weak solutions of parabolic equations. 
	%To appear in	Proc. Roy. Soc. Edinburgh Sect. A.
			{\it Proc. Roy. Soc. Edinburgh Sect. A}, to appear.
	}
	
	% Email address of each of all authors is required.
	% You may list email addresses of all other authors, separately.
	\email{xxu@math.msstate.edu}
	%\email{email2@aimSciences.org}
	%\email{email3@ece.pdx.edu}
	
	% Put your short thanks below. For long thanks/acknowlegements,
	%please go to the last acknowlegments section.
	%\thanks{The first author is supported by NSF grant xx-xxxx}
	
	% Add corresponding author at the footnote of the first page if it is necessary. 
	% Plase add $^*$ adjacent to the corresponding author's name on the first page. 
	% The example shown in this template is if the first author is the corresponding author.
	%\thanks{$^*$ Corresponding author: xxxx}
	
	\maketitle

	% Enter the first author's name and address:
	\centerline{\scshape Xiangsheng Xu}
	\medskip
	{\footnotesize
		% please put the address of the first author
		\centerline{Department of Mathematics \& Statistics}
		\centerline{Mississippi State University}
		\centerline{ Mississippi State, MS 39762, USA}
	} % Do not forget to end the {\footnotesize by the sign }

	\bigskip

	% The name of the associate editor will be entered by an editorial staff
	% "Communicated by the associate editor name" is not needed for special issue.
	% \centerline{(Communicated by the associate editor name)}
	
	\begin{abstract}
		It is well known that a weak solution $\varphi$ to the initial boundary value problem for 
		the uniformly parabolic equation $\partial_t\varphi-\mbox{div}(A\nabla \varphi) +\omega\varphi= f $ in
		$\Omega_T\equiv\Omega\times(0,T)$ satisfies the uniform estimate
		$$
		\|\varphi\|_{\infty,\Omega_T}\leq \|\varphi\|_{\infty,\partial_p\Omega_T}+c\|f\|_{q,\Omega_T}, \ \ \
		c=c(N,\lambda, q,  \Omega_T),
		$$
provided that $q>1+\frac{N}{2}$, where $\Omega$ is a bounded domain in $\RN$ with
			Lipschitz boundary, $T>0$, $\partial_p\Omega_T$ is the parabolic boundary of $\Omega_T$, $\omega\in L^1(\Omega_T)$ with $\omega\geq 0$, and $\lambda$ is the smallest eigenvalue of the coefficient matrix $A$. This estimate is sharp in the sense that it generally fails if  $q=1+\frac{N}{2}$. In this paper we show that the linear growth of this upper bound in $\|f\|_{q,\Omega_T}$ can be improved. To be precise, we establish
			\begin{equation*}
			\|\varphi\|_{\infty,\Omega_T}\leq
			\|\varphi_0\|_{\infty,\partial_p\Omega_T}+c\|f\|_{1+\frac{N}{2},\Omega_T}\left(\ln(\|f\|_{q,\Omega_T}+1)+1\right).
			\end{equation*}	
	\end{abstract}

\section{Introduction}
Let $\Omega$ be a bounded domain in $\RN$ with Lipschitz
boundary $\partial \Omega$. 
%Assume that $\Gamma_D$ is a non-empty
%open subset of $\partial \Omega$. 
For each $T>0$ consider the initial boundary value
problem
\begin{eqnarray}
\partial_t\varphi-\mbox{div}(A\nabla \varphi)+\omega \varphi & = & f\ \ \mbox{in}\ \
\Omega_T\equiv\Omega\times(0,T),\label{par1}\\
\varphi & = & 0 \ \ \mbox{on $\Sigma_T\equiv\partial\Omega\times(0,T)$,}\label{par2}\\
\varphi(x,0)
&=& \varphi_0(x) \ \ \mbox{on $\Omega$.}\label{par3}
\end{eqnarray}
We assume:
\begin{enumerate}
	\item[(H1)] $A=A(x,t)$ is an $N\times N$ matrix whose entries $a_{ij}(x,t)$
	satisfy
	\begin{equation}
	a_{ij}(x,t)\in L^\infty(\ot),\ \ \lambda|\xi|^2\leq \left(A(x,t)\xi\cdot\xi\right)\
	\ \ \mbox{for $\xi\in \RN$ and a.e. $(x,t)\in\ot$},
	\end{equation} where $\lambda >0$;
	\item[(H2)] $\omega \in L^{1}(\ot)$ with $\omega\geq 0$ a.e. on $\ot$, $f\in L^{q}(\ot)$ for some $q>1+\frac{N}{2}$, and
	$\varphi_0\in L^\infty(\Omega)$.
\end{enumerate}
%Here  As for other given data in the problem we have that 
%, and $\nu$ is the unit outward normal to $\partial \Omega$.
In the situation considered here, the classical theory for uniformly parabolic equations \cite{LSU}
asserts that there is a unique weak solution $\varphi$ in the space
$L^2(0,T;W^{1,2}_0(\Omega))\cap L^\infty(\ot)$ and we have the estimate 
%$\varphi$ satisfies
\begin{equation}
\|\varphi\|_{\infty,\ot}\leq \|\varphi_0\|_{\infty,\Omega}+c\|f\|_{q,\ot}, \ \ \
c=c(N,\lambda, q,  \ot).
\end{equation}
%Here and in what follows we denote by $\|f\|_p$ the norm of $f$ in $L^p(\Omega)$. 
The result is sharp in the sense that if $q=1+\frac{N}{2}$
then the above inequality fails in general. Our objective here is to improve
this upper bound for $\|\varphi\|_{\infty,\ot}$. 
To be precise, we will show:
\begin{theorem}\label{logb} Let \textup{(H1)-(H2)} be satisfied and $\varphi$ be a weak solution to
	\eqref{par1}-\eqref{par3} in the space $L^2(0,T;W^{1,2}_0(\Omega))$.
	%Under the above assumptions
	Then there is a positive number $c=c(N,\lambda, q, \ot)$ such
	that
	\begin{equation}\label{ppb}
	\|\varphi\|_{\infty,\ot}\leq
	\|\varphi_0\|_{\infty,\Omega}+c\|f\|_{1+\frac{N}{2},\ot}\left(\ln(\ftq+1)+1\right).
	\end{equation}	
\end{theorem}

Similar results for functions in $W^{s,q}(\RN), sq>N,$ have been established
in \cite{BW,KT}. This theorem can be viewed as the parabolic version of the result in \cite{X1}. As in \cite{X1}, we introduce a change of variable.  The equation satisfied by the new unknown function $v$
has the expression %there appears the term $gv$. 
\begin{eqnarray}
	\partial_tv-\mbox{div}(A\nabla v) +\frac{1}{v}A\nabla v\cdot\nabla
v+\omega v\ln v& = & gv\ \ \mbox{in $\ot$.}
\label{eqnv1}
\end{eqnarray}
%See \eqref{eqnv1} below. 
To prove Theorem \ref{logb}, we need to have an estimate like
\begin{equation}\label{18fm1}
\supe_{\ot} |v|\leq c\left(\|g\|_{q,\ot}+1\right)^\alpha,\ \ c, \alpha>0,
\end{equation} 
%$\supe_{\ot} |v$
%be bounded by a power function of $\|g\|_{q,\ot}+1$, 
where $q$ is given as in (H2). 
%Obviously, the difficulty is due to 
%To do this, we must overcome the difficulty caused by 
%the fact that
	%Before we continue, we would like to remark that 
%	Since 
	%	In the parabolic case, t
The approach in (\cite{LSU}, p.185) is to first show the boundedness of $v$ for small $t$ and then extend the result to large $t$. Unfortunately, 
this method does not serve our purpose because
%the same issue arises. That is, 
how $\supe_{\ot}|v|$ is bounded by $\|g\|_{q,\ot}$ becomes unclear. By modifying Moser's technique of iteration of $L^q$ norms (see \cite{M}), we are able to obtain \eqref{18fm1}, i.e., 
%that
	%an explicit dependence of 
	$\supe_{\ot}|v|$ is bounded by a power function of $\|g\|_{q,\ot}+1$. 
	%This is the essential part of our proof of Theorem \ref{logb}.  
	Even though our proof here is inspired by
	the work of \cite{X1} in the elliptic case, we must overcome the complication caused by the time variable. This constitutes the core of our analysis.
	
	Observe that since the coefficient function $``g"$  in equation \eqref{eqnv1}
%	in \eqref{eqnv1} 
	is of arbitrary sign  this  equation, in essence, resembles the parabolic version of the Schr\"{o}dinger equation in \cite{CFG}. In the elliptic case \cite{CFG}, one can assume that $g$ belongs to the 
	slightly more general Kato class, instead of $L^p(\Omega)$ with $p>\frac{N}{2}$. However, this results in the loss of explicit dependence of $\supe|v|$ on $g$. That is, how  $\supe|v|$ is bounded by a certain norm of $g$ is hidden. This explains why an elliptic version of Theorem \ref{logb} is established in \cite{X1} only under the assumption that $g\in L^p(\Omega)$ with $p>\frac{N}{2}$.
	 %In particular, w
	 
	 We also would like to point out that the constant $c$ in \eqref{ppb} does not depend on the upper bounds of our elliptic coefficients $a_{ij}(x,t)$. This opens the possibility that Theorem \ref{logb} be
	 %is also 
	 applicable to certain types of
the so-called singular parabolic equations.

Our result can also be extended to more general equations by suitably modifying our proof. The inequality \eqref{ppb} can also have different versions. For example, the interested reader might want to try to establish the following
%One such example is 
	\begin{equation}\label{ppbb}
\|\varphi\|_{\infty,\ot}\leq
\|\varphi_0\|_{\infty,\Omega}+c\sup_{0\leq t\leq T}\|f\|_{\frac{N}{2},\Omega}\left(\ln(\sup_{0\leq t\leq T}\|f\|_{q,\Omega}+1)+1\right),\ \ \ q>\frac{N}{2}.
\end{equation}	
We refer the reader to \cite{LSU} for more information on how to balance the integrability of $f$ in the time variable and that of the space variables.
%We encourage the interested reader to infer the above inequality from the proof of Theorem
%\ref{logb}.
%do not have to be bounded.
%We may view our result here the PDE version of that
%Let $\Omega$ be a bounded domain in $\mathbb{R}^N$ and $T$ a positive number. Set $\ot=\Omega\times(0,T)$. We study the behavior of solutions of the system

The rest of the paper is dedicated to the proof of Theorem \ref{logb}. In our proof we will assume that the space dimension $N$ is bigger than $2$ due to an application of the Sobolev embedding theorem. Obviously, if $N=2$, we must take the following version of the the Sobolev embedding theorem: for each $s>2$ there is a positive number $c_s$ such that
\begin{eqnarray}
\int_{\Omega_T}|v|^{2+\frac{2(s-2)}{s}}dxdt&\leq& %\int_{0}^{T}\left(\int|v|^sdx\right)^{\frac{2}{s}}\left(\int_{\Omega}|v|^2dx\right)^{\frac{s-2}{s}}dt\nonumber\\
%&\leq& 
c_s^2\left(\sup_{0\leq t\leq T}\int_{\Omega}|v|^2dx\right)^{\frac{s-2}{2}}\int_{\Omega_T}|\nabla v|^{2}dxdt
\end{eqnarray}
for each $v\in L^\infty(0,T; L^2(\Omega))\cap L^2(0,T; W^{1,2}_0(\Omega))$.
Then all our arguments in the proof of Theorem \ref{logb} remain valid, i.e., Theorem \ref{logb} still holds for $N=2$.
% via a slight modification of our proof.
\section{Proof of Theorem \ref{logb}}

%result.
\begin{proof}[Proof of Theorem \ref{logb}] We decompose $\varphi$ into
	$\varphi_1+\varphi_2$, where $\varphi_1$ and $\varphi_2$ are the
	respective solutions of the following two problems
	\begin{eqnarray}
	\partial_t\varphi_1-\mbox{div}(A\nabla \varphi_1)+\omega \varphi_1 & = & f\ \ \mbox{in}\ \
	\ot,\label{phi11}\\
	\varphi_1 & = & 0 \ \ \mbox{on $\Sigma_T$,}\label{phi12}\\
	\varphi_1(x,0)
	&=& 0 \ \ \mbox{on $\Omega$ and}\label{phi13}
	\end{eqnarray}
	\begin{eqnarray}
	\partial_t\varphi_2-\mbox{div}(A\nabla \varphi_2)+\omega \varphi_2 & = & 0\ \ \mbox{in}\ \
	\ot,\label{phi21}\\
	\varphi_2 & = & 0 \ \ \mbox{on $\Sigma_T$,}\label{phi22}\\
	\varphi_2(x,0)
	&=& \varphi_0(x) \ \ \mbox{on $\Omega$.}\label{phi23}
	\end{eqnarray}
	Let $c_s$ be the
	smallest positive number such that
	\begin{equation}\label{sin}
	\|\phi\|_{\frac{2N}{N-2},\Omega}\leq c_s\|\nabla \phi\|_{2,\Omega}
	\end{equation}
	for all $\phi\in W^{1,2}(\Omega)$ with $\phi=0$ on $\partial\Omega$. It is
	well known that the Sobolev constant $c_s$ here depends only on
	$N$ (\cite{EG}, p. 138) . 
	%If 
%	\begin{equation}
%	\|f\|_{1+\frac{N}{2},\ot}> 1,
%	\end{equation}
	We consider the functions
	\begin{eqnarray}
	u &=& \frac{\varphi_1}{\max\{\|f\|_{1+\frac{N}{2},\ot},1\}},\label{key}\\
	g &=&\frac{f}{\max\{\|f\|_{1+\frac{N}{2},\ot},1\}}.\label{gn}
	\end{eqnarray}
	Obviously, they satisfy
	\begin{eqnarray}
	\partial_tu-\mbox{div}(A\nabla u)+\omega u & = & g\ \ \mbox{in}\ \
	\ot,\label{eqnu1}\\
	u & = & 0 \ \ \mbox{on $\Sigma_T$,}\label{eqnu2}\\
	u(x,0)
	&=& 0 \ \ \mbox{on $\Omega$.}\label{eqnu3}
	\end{eqnarray}
%	Furthermore,
%	\begin{equation}\label{gn1}
%	\|g\|_{1+\frac{N}{2},\ot}\leq 1.
%	\end{equation}
	
	\begin{lemma}\label{expb}Let the assumptions of Theorem \ref{logb} hold.
	%	Assume that $N>2$. Then t
		Then to each $\alpha>0$ sufficiently small
		%\in (0,\frac{4\lambda}{c_s^2})$
		there corresponds a positive number $c=c(N,\alpha, c_s, \lambda, \ot, \|f\|_{1+\frac{2}{N},\ot})$ such that
		\begin{equation}\label{expb1}
		\ioT e^{\alpha \left(1+\frac{2}{N}\right)u}dxdt\leq c.
	%	\sup_{0\leq t\leq T}\io e^{\alpha u}dx+\ioT\left|\nabla e^{\frac{\alpha}{2}u}\right|^2dxdt\leq c\ioT |g|dxdt+c.
		%\int_\Omega \exp{\left(\frac{\alpha Nu}{N-2} \right)} dx\leq c.
		\end{equation}	
	\end{lemma}
	\begin{proof}
	%	Let $\alpha\in (0, \frac{4\lambda}{c_s^2})$ be
	%	given. 
	First we would like to remark that a weak solution $u$ to \eqref{eqnu1}-\eqref{eqnu3} is unique \cite{LSU} and can be constructed as the limit of a sequence of smooth approximate solutions, which can be obtained by, for example, regularizing the given functions in the problem. Thus without any loss of generality, we may assume that $u$ is a classical solution in our subsequent calculations.
	Let $\alpha>0$ be given.	Using $e^{\alpha u}-1$ 
		%is  , and upon using it legitimate 
		a test
		function in \eqref{eqnu1}, which means that we multiply through the equation by the function and integrate the resulting equation over $\Omega_s\equiv\Omega\times(0,s)$, $T\geq s>0$,  we obtain
		\begin{equation} \label{gn4}
		\io\left(\frac{e^{\alpha u}}{\alpha}-u\right)dx+\frac{4\lambda}{\alpha}\int_{\Omega_s}\left|\nabla e^{\frac{\alpha}{2}u}\right|^2dxdt+\int_{\Omega_s}\omega u\left(e^{\alpha u}-1\right)dxdt\leq \int_{\Omega_s}g(e^{\alpha u}-1)dxdt+\frac{|\Omega|}{\alpha}.\end{equation}
		Note that 
		\begin{eqnarray}
		u\left(e^{\alpha u}-1\right)&\geq& 0,\nonumber\\
		e^{\alpha u}-1
		&=&\left(e^{\frac{\alpha u}{2}}-1\right)^2+2\left(e^{\frac{\alpha
				u}{2}}-1\right),\label{gn2}\\
			\|g\|_{1+\frac{N}{2},\Omega_T}&\leq&1.\label{gn1}
	%	\sup_{0\leq t\leq T}	\|g\|_{\frac{2N}{N+2},\Omega}&=&\frac{\sup_{0\leq t\leq T}\|f\|_{\frac{2N}{N+2},\Omega}}{\|f\|_{1+\frac{N}{2},\ot}}\leq\Omega|^{\frac{N-2}{2N}}.
		\end{eqnarray} 
		Keeping these in mind, we estimate from the
		Sobolev embedding theorem \eqref{sin} that
		\begin{eqnarray}
			\int_{0}^{T}\io\left(e^{\frac{\alpha u}{2}}-1\right)^{2+\frac{4}{N}}dxdt&\leq&	\int_{0}^{T}\left(\io\left(e^{\frac{\alpha u}{2}}-1\right)^{\frac{2N}{N-2}}dx\right)^{\frac{N-2}{N}}\left(\io\left(e^{\frac{\alpha u}{2}}-1\right)^{2}dx\right)^{\frac{2}{N}}dt\nonumber\\
			&\leq &c_s^2\left(\max_{0\leq t\leq T}\io\left(e^{\frac{\alpha u}{2}}-1\right)^{2}dx\right)^{\frac{2}{N}}	\int_{0}^{T}\io|\nabla e^{\frac{\alpha u}{2}}|^2dxdt\nonumber\\
				&\leq &c_s^2\left(\max_{0\leq t\leq T}\io e^{\alpha u}dx\right)^{\frac{2}{N}}	\int_{0}^{T}\io|\nabla e^{\frac{\alpha u}{2}}|^2dxdt\nonumber\\
				&&+c_s^2|\Omega|^{\frac{2}{N}}\int_{0}^{T}\io|\nabla e^{\frac{\alpha u}{2}}|^2dxdt\nonumber\\
				&\leq &\frac{c_s^2\alpha^{1+\frac{2}{N}}}{4\lambda}\left(\int_{\Omega_T}g(e^{\alpha u}-1)dxdt+\max_{0\leq t\leq T}\io udx+\frac{|\Omega|}{\alpha}\right)^{1+\frac{2}{N}}\nonumber\\
					&&+\frac{c_s^2|\Omega|^{\frac{2}{N}}\alpha}{4\lambda}\left(\int_{\Omega_T}g(e^{\alpha u}-1)dxdt+\max_{0\leq t\leq T}\io udx+\frac{|\Omega|}{\alpha}\right)\nonumber\\
					&\leq &\left(\frac{c_s^2\alpha^{1+\frac{2}{N}}}{4\lambda}+\delta\right)\left(\int_{\Omega_T}|g(e^{\alpha u}-1)|dxdt+\max_{0\leq t\leq T}\io |u|dx+\frac{|\Omega|}{\alpha}\right)^{1+\frac{2}{N}}\nonumber\\
					&&+c(\delta),\ \ \ \delta>0.\label{gn3}
		\end{eqnarray}
	With the aid of \eqref{gn1} and \eqref{gn2}, we  estimate 
	%the integral on the right-hand side of \eqref{gn4}  as follows:
			\begin{eqnarray}
			\int_{\Omega_T}|g(e^{\alpha u}-1)|dxdt&\leq&
				\int_{\Omega_T}|g|\left(e^{\frac{\alpha
					u}{2}}-1\right)^2dxdt+2	\int_{\Omega_T}
			|g||e^{\frac{\alpha u}{2}}-1|dxdt\nonumber\\
			&\leq
			&\|g\|_{1+\frac{N}{2},\Omega_T}\left(	\int_{\Omega_T}\left(e^{\frac{\alpha
					u}{2}}-1\right)^{\frac{2(N+2)}{N}}dxdt\right)^{\frac{N}{N+2}}\nonumber\\
			&&
			+2\|g\|_{\frac{2(N+2)}{N+4},\Omega_T}\left(	\int_{\Omega_T}\left(e^{\frac{\alpha
					u}{2}}-1\right)^{\frac{2(N+2)}{N}}dxdt\right)^{\frac{N}{2(N+2)}}\nonumber\\
	&\leq&\|e^{\frac{\alpha
			u}{2}}-1\|^2_{\frac{2(N+2)}{N},\ot}
	+2|\Omega_T|^{\frac{N}{2(N+2)}}\|e^{\frac{\alpha
			u}{2}}-1\|_{\frac{2(N+2)}{N},\ot}\nonumber\\
		&\leq&2\|e^{\frac{\alpha
			u}{2}}-1\|^2_{\frac{2(N+2)}{N},\ot}	+c.
			\end{eqnarray}
			Plugging this into \eqref{gn3}, we obtain
			\begin{eqnarray}
				\int_{0}^{T}\io\left(e^{\frac{\alpha u}{2}}-1\right)^{2+\frac{4}{N}}dxdt
				%&\leq&
			%2^{1+\frac{4}{N}}\left(\frac{c_s\alpha^{1+\frac{2}{N}}}{4\lambda}+\delta\right)\left(\|e^{\frac{\alpha
			%		u}{2}}-1\|^2_{\frac{2(N+2)}{N},\ot}\right)^{1+\frac{2}{N}}\nonumber\\
		%	&&+c\left(\max_{0\leq t\leq T}\io |u|dx\right)^{1+\frac{2}{N}}+c\nonumber\\
			&\leq&
			2^{1+\frac{4}{N}}\left(\frac{c_s\alpha^{1+\frac{2}{N}}}{4\lambda}+\delta\right)\int_{0}^{T}\io\left(e^{\frac{\alpha u}{2}}-1\right)^{2+\frac{4}{N}}dxdt\nonumber\\
			&&+c\left(\max_{0\leq t\leq T}\io |u|dx\right)^{1+\frac{2}{N}}+c.
			\end{eqnarray}
		Thus choosing $\alpha, \delta$ suitably small, 
		%we get
		we can absorb the first term on the right-hand side to the left-hand side to obtain
		\begin{equation}\label{gn6}
		\int_{0}^{T}\io\left(e^{\frac{\alpha u}{2}}-1\right)^{2+\frac{4}{N}}dxdt\leq c\left(\max_{0\leq t\leq T}\io |u|dx\right)^{1+\frac{2}{N}}+c.
		\end{equation}
		%Integrate with respect to $t$ to derive
	%	\begin{eqnarray}
	%	\frac{1}{\alpha}\io e^{\alpha u}dx+\left(\frac{4\lambda}{\alpha}-c_s^{2}\right)\iot\left|\nabla e^{\frac{\alpha}{2}u}\right|^2dxdt&\leq &\io udx+2c_s|\Omega|^{\frac{N-2}{2N}}\tau^{\frac{1}{2}}\left(\iot\left|\nabla e^{\frac{\alpha}{2}u}\right|^2dxdt\right)^{\frac{1}{2}}\nonumber\\
%		&\leq &\io udx+\delta\iot\left|\nabla e^{\frac{\alpha}{2}u}\right|^2dxdt+c(\delta),
		\label{expi1}
		%\end{eqnarray}
	%	where $\Omega_\tau=\Omega\times(0,\tau), \tau\in (0, T]$ and $\delta\in (0, \frac{4\lambda}{\alpha}-c_s^{2})$.
		To estimate the right-hand side of the above inequality, we  define, for $\varepsilon>0$, 
		$$\eta_{\varepsilon}(s)=\left\{\begin{array}{ll}
		1& \mbox{if $s\geq\varepsilon$,}\\
		\frac{1}{\varepsilon}s& \mbox{if $-\varepsilon< s<\varepsilon$,}\\
		-1& \mbox{if $s\leq-\varepsilon$.}
		\end{array}\right.
		$$
		Use $\eta_{\varepsilon}(u)$ as a test function in \eqref{eqnu1} to
		obtain
		\begin{equation}
		\frac{d}{dt}\io\int_{0}^{u}\eta_{\varepsilon}(s)dsdx\leq\io g\eta_{\varepsilon}(u)dx\leq\io|g|dx.
		\end{equation}
		Integrate with respect to $t$ and then let $\varepsilon\rightarrow 0^+$ to yield
		\begin{equation}\label{extr}
		\sup_{0\leq t\leq T}\io |u|dx\leq\ioT|g|dxdt.
		\end{equation}
		Plugging this into \eqref{gn6} gives the desired result.
	\end{proof}
	
	To continue the proof of Theorem \ref{logb}, we assume, without any loss of
	generality, that
	% \eqref{key} holds and 
	\begin{equation}
	\supe_{\ot} u
	= \|u\|_{\infty,\ot}.\label{pmu}
	\end{equation}
	Indeed, if \eqref{pmu} is not true, we multiply through \eqref{eqnu1} by $-1$
	and consider $-u$. 
	%If \eqref{key} is not true, we can prove Lemma \ref{expb} for
	%\eqref{phi11}-\eqref{phi13}. That is, 
	%directly because 
	%we  have that \eqref{expb1} holds for
%	$u=\varphi_1$. 
	%If $N=2$, we deduce from (D1) and Poincar\'{e}'s inequality that
	%\begin{equation}
	%\avint_{B_r(y)}
	%\end{equation}
	Let \begin{equation} v=e^u.\end{equation} Then
	\begin{eqnarray*}
		u &=& \ln v,\\
		A\nabla u &=& \frac{1}{v}A\nabla v,\\
		\mbox{div}(A\nabla u) &=& -\frac{1}{v^2}A\nabla v\cdot\nabla
		v+\frac{1}{v}\mbox{div}(A\nabla v).
	\end{eqnarray*}
	Substitute these into \eqref{eqnu1}-\eqref{eqnu3} to obtain \eqref{eqnv1} coupled with
	\begin{eqnarray}
	v& = & 1 \ \ \mbox{on $\Sigma_T$},\label{eqnv2}\\
	v(x,0) &=& 1 \ \ \mbox{on $\Omega$}.\label{eqnv3}
	\end{eqnarray}
	%	We employ a technique of iteration of $L^q$ norms originally due to Moser \cite{M}. 
	%Without loss of generality, assume This is the key difference between To do this
	%\begin{equation}
%	\|u^+\|_\infty=\|u\|_\infty.
%	\end{equation}
%	Set $a=(\|f\|_q)^{\frac{1}{p-1}}$. For each $s> \frac{p-1}{p}$ we use $\frac{s^p}{ps-p+1}(u^++a)^{ps-p+1}$ as a test function in \eqref{of1} to obtain
	%The above formal calculations can easily be made vigorous.
	%To continue our proof, we s
	Set
	\begin{equation}\label{wp1}
	w=\max\{v,1\}.
	\end{equation}
%	Obviously, we have
%	\begin{equation}
%	w(x,t)=\left\{\begin{array}{ll}
%	v(x,t) & \mbox{if $v(x,t)>1$.}\\
	%1 &\mbox{if $v(x,t)\leq 1$}\end{array}
%	\right. \end{equation}
	%Fix
%	In particular, there holds
%	\begin{equation}\label{wp2}
	%w\geq 1\ \ \mbox{a.e. on $\ot$.}
	%\end{equation}
	For each $\beta>0$ we use $w^\beta-1$ as a test function in \eqref{eqnv1} to
	obtain
	\begin{eqnarray}
	\lefteqn{\frac{d}{dt}\io\int_{0}^{v}(((s-1)^++1)^\beta-1)dsdx}\nonumber\\
	&&+\int_\Omega \beta w^{\beta-1}A(x)\nabla v\cdot\nabla w
	dx
	\leq\int_\Omega gv(w^{\beta }-1)dx
	\leq\int_\Omega|g|w^{\beta+1}dx.
	%\leq\sup_{0\leq t\leq T}\|g\|_{q,\Omega}\|w^{\beta+1}\|_{\frac{q}{q-1},\Omega}.
	\label{log1}
	\end{eqnarray} 
	%Then it is easy to see that and  \eqref{wp2}
	%Observe that
	With \eqref{wp1}  in mind, we can deduce that
	\begin{eqnarray}
	\int_{0}^{v}(((s-1)^++1)^\beta-1)ds&=&\left(\int_{1}^{v}(((s-1)^++1)^\beta-1)ds\right)^+\nonumber\\
	&=&\left(\int_{1}^{v}(s^\beta-1)ds\right)^+\nonumber\\
	&=&\frac{1}{\beta+1}\left(w^{\beta+1}-1\right)-\left(w-1\right).
	\end{eqnarray}
	Using this in \eqref{log1} yields
	\begin{eqnarray}\label{log11}
	\lefteqn{\frac{1}{\beta+1}\io(w^{\beta+1}-1)dx+\frac{4\beta\lambda}{(\beta+1)^2}\iot\left|\nabla w^{\frac{\beta+1}{2}}\right|^2dxdt}\nonumber\\
	&\leq&\iot|g|w^{\beta+1}dxdt+\io(w-1)dx.
	\end{eqnarray}
	%Here we have used the fact that
	%\begin{equation}
	%v(w^{\beta }-1)\leq w^{\beta+1}.
	%\end{equation}
	Let $\eta_{\varepsilon}$ be given as before. We use $\eta_{\varepsilon}(v-1)$ as a test function in \eqref{eqnv1}.
	With the aid of the proof of \eqref{extr}, we arrive at
	\begin{eqnarray}
	\io(w-1)dx=\io(v-1)^+dx\leq\io|v-1|dx\leq\ioT|g|dxdt\leq\ioT|g|w^{\beta+1}dxdt.
	\end{eqnarray}
	The last step is due to the fact that $w\geq 1$ on $\ot$.
	%\eqref{wp2}.
	Plug this into \eqref{log11} to yield
	\begin{equation}\label{log12}
	\frac{1}{\beta+1}\sup_{0\leq t\leq T}\io(w^{\beta+1}-1)dx+\frac{4\beta\lambda}{(\beta+1)^2}\ioT\left|\nabla w^{\frac{\beta+1}{2}}\right|^2dxdt\leq2\ioT|g|w^{\beta+1}dxdt.
	\end{equation}
	Note that
	\begin{eqnarray*}
		w^{\beta+1} -1&=& \left(w^{\frac{\beta+1}{2}}-1+1\right)^2-1\\
		&\leq & 2\left(w^{\frac{\beta+1}{2}}-1\right)^2+1. \end{eqnarray*}
	Keeping this, \eqref{wp1}, the Sobolev embedding
	theorem \eqref{sin}, and \eqref{log12} in mind,  we calculate
	%, with the aid of , 
	that
	\begin{eqnarray}
	\lefteqn{	\int_{0}^{T}\io w^{(\beta+1)(1+\frac{2}{N})}dxdt}\nonumber\\
	&\leq&
	2^{\frac{2}{N}}\int_{0}^{T}\io \left(w^{\beta+1}-1\right)^{1+\frac{2}{N}}dxdt+2^{\frac{2}{N}}\int_{0}^{T}\io w^{\beta+1}dxdt\nonumber\\
	&\leq&
	2^{\frac{2}{N}}\left(\sup_{0\leq t\leq T}\io(w^{\beta+1}-1)dx\right)^{\frac{2}{N}}\int_{0}^{T}\left(\io\left(w^{\beta+1}-1\right)^{\frac{N}{N-2}}dx\right)^{\frac{N-2}{N}}dt\nonumber\\
	&&+2^{\frac{2}{N}}\int_{0}^{T}\io w^{\beta+1}dxdt\nonumber\\
	&\leq&c(\beta+1)^{\frac{2}{N}}\left(\ioT|g|w^{\beta+1}dxdt\right)^{\frac{2}{N}}\int_{0}^{T}\left(\io\left(w^{\frac{\beta+1}{2}}-1\right)^{\frac{2N}{N-2}}dx\right)^{\frac{N-2}{N}}dt\nonumber\\
	&&+c(\beta+1)^{\frac{2}{N}}\left(\ioT|g|w^{\beta+1}dxdt\right)^{\frac{2}{N}}+2^{\frac{2}{N}}\ioT w^{\beta+1}dxdt\nonumber\\
	&\leq&c\frac{(\beta+1)^{\frac{2}{N}+2}}{\beta}\left(\ioT|g|w^{\beta+1}dxdt\right)^{\frac{2}{N}+1}\nonumber\\
	&&+c(\beta+1)^{\frac{2}{N}}\left(\ioT|g|w^{\beta+1}dxdt\right)^{\frac{2}{N}}+2^{\frac{2}{N}}\ioT w^{\beta+1}dxdt.\label{log13}
	\end{eqnarray}
	%Note that h
	%Here 
	%$c$ is independent of $T$ and
	%we have used the fact that
	%$w\geq 1$ on $\ot$. Continuing to utilize this fact, 
	%Fix $q>1+\frac{N}{2}$. 
	We estimate from \eqref{wp1} that
	\begin{eqnarray}
	\left(\ioT w^{\beta+1}dxdt\right)^{\frac{N}{N+2}}&\leq&\left[\left(\ioT w^{\frac{(\beta+1)q}{q-1}}dxdt\right)^{\frac{q-1}{q}}|\ot|^{\frac{1}{q}}\right]^{\frac{N}{N+2}}\nonumber\\
	&=&\left[\left(\avint_{\ot} w^{\frac{(\beta+1)q}{q-1}}dxdt\right)^{\frac{q-1}{q}}|\ot|\right]^{\frac{N}{N+2}}\nonumber\\
	&\leq&|\ot|^{\frac{N}{N+2}-\frac{q-1}{q}}\left(\int_{\ot} w^{\frac{(\beta+1)q}{q-1}}dxdt\right)^{\frac{q-1}{q}}.
	\end{eqnarray}
	Recall that $q>1+\frac{N}{2}$ is given by assumption (H2).
	%where $q>2$ is to be determined. 
	Raising both sides of \eqref{log13} to the $\frac{N}{N+2}$ -th power, we derive that
	\begin{eqnarray*}
		\|w^{\beta+1}\|_{\frac{N+2}{N}, \ot} &\leq &\frac{c(\beta+1)^{\frac{2(N+1)}{N+2}}}{
			\beta^{\frac{N}{N+2}}}\|g\|_{q,\ot}\left(\ioT
		w^{\frac{(\beta+1)q}{q-1}}dxdt\right)^{\frac{q-1}{q}}\nonumber\\
		&&+c(\beta+1)^{\frac{2}{N+2}}|\ot|^{\frac{2(q-1)}{q(N+2)}-\frac{q}{q-1}}\|g\|_{q,\ot}^{\frac{2}{N+2}}\left(\ioT w^{\frac{(\beta+1)q}{q-1}}dxdt\right)^{\frac{q-1}{q}}\nonumber\\
		&&+2^{\frac{2}{N}}|\ot|^{\frac{N}{N+2}-\frac{q-1}{q}}\left(\int_{\ot} w^{\frac{(\beta+1)q}{q-1}}dxdt\right)^{\frac{q-1}{q}}\nonumber\\
		&\leq&\left(c(\beta+1)^{\frac{2(N+1)}{N+2}}\left(\frac{1}{
			\beta^{\frac{N}{N+2}}}+1\right)\|g\|_{q,\ot}+c\right)\|w^{\beta+1}\|_{\frac{q}{q-1},\ot}\nonumber\\
		&\leq&c(\beta+1)^{\frac{2(N+1)}{N+2}}\left(\frac{1}{
			\beta^{\frac{N}{N+2}}}+2\right)\left(\|g\|_{q,\ot}+1\right)\|w^{\beta+1}\|_{\frac{q}{q-1},\ot}.
	\end{eqnarray*}
	%The last step is due to \eqref{gn}.
	Thus we can write the above inequality in the form
	\begin{equation}
	\|w\|_{\frac{(N+2)(\beta+1)}{N}} \leq c^{\frac{1}{\beta+1}}(\beta+1)^{\frac{2(N+1)}{(N+2)(\beta+1)}}\left(\frac{1}{
		\beta^{\frac{N}{N+2}}}+2\right)^{\frac{1}{\beta+1}}\left(\|g\|_{q,\ot}+1\right)^{\frac{1}{\beta+1}}\|w\|_{\frac{q(\beta+1)}{q-1}}.
	\end{equation}
	Now set $$
	%\begin{eqnarray*}
	\chi  =\frac{(N+2)/N}{q/(q-1)}>1.$$ Fix $\beta_0>0$ and let $$
	\beta+1 = (1+\beta_0)\chi^i, \ \  i=0,1,2,\cdots.$$
	%\end{eqnarray*}
	Subsequently,
	\begin{eqnarray*}
		\|w\|_{(1+\beta_0)\frac{q}{q-1}\chi^{i+1},\ot} &\leq &
		\left[c(\|g\|_{q,\ot}+1)\right]^{\frac{1}{(1+\beta_0)\chi^i}}
		\left[(1+\beta_0)\chi^{i}\right]^{\frac{2(N+1)}{(1+\beta_0)(N+2)\chi^i}}\nonumber\\
		&&\cdot\|w\|_{(1+\beta_0)\frac{q}{q-1}\chi^{i},\ot}\\
		&\leq
		&\left[c(\|g\|_{q,\ot}+1)\right]^{\frac{1}{(1+\beta_0)}(\frac{1}{\chi^i}+\cdots
			+1)}(1+\beta_0)^{\frac{2(N+1)}{(1+\beta_0)(N+2)}(\frac{1}{\chi^i}+\cdots
			+1)}\\
		& & \cdot \chi^{\frac{2(N+1)}{(1+\beta_0)(N+2)}(\frac{i}{\chi^i}+\cdots
			+1)}\|w\|_{(1+\beta_0)\frac{q}{q-1},\ot}.
	\end{eqnarray*}
	Taking $i\rightarrow \infty$ yields
	\begin{equation}\label{log14}
	\|w\|_{\infty,\ot}\leq
	c\left(\|g\|_{q,\ot}+1\right)^{\alpha_0}\|w\|_{r,\ot},
	\end{equation}
	where 
	\begin{eqnarray}
	\alpha_0&=&\frac{\chi}{(1+\beta_0)(\chi-1)},\\
	r&=&(1+\beta_0)\frac{q}{q-1}.
	\end{eqnarray}
	%$\alpha_0$ is a positive number depending on $\beta_0, q$.
%	Let
%	$$.$$ 
	In view of Lemma \ref{expb}, we can  find an $\alpha\in(0,r)$ small enough so that
	\begin{equation}\label{log25}
	\ioT w^\alpha dxdt\leq	\ioT e^{\alpha u}dxdt+c\leq c.
	\end{equation}
%	$$r< \frac{4\lambda}{c_s^2}.$$ 
%We may assume that $r\leq\alpha$. If not, 
%pick a number $\alpha$ so
%	that \begin{equation}0<\alpha<\frac{4\lambda}{c_s^2}\leq
%	r.\end{equation} 
%then we have
It follows that
	\begin{eqnarray*}
		\|w\|_{r,\ot} &=&\left(\ioT
		w^{r-\alpha}w^\alpha dx\right)^{\frac{1}{r}}\\
		&\leq &
		\|w\|_{\infty,\ot}^{\frac{r-\alpha}{r}}\|w\|_{\alpha,\ot}^{\frac{\alpha}{r}}.
	\end{eqnarray*}
	This together with  \eqref{log14} gives
	\begin{equation}
	\|w\|_{\infty,\ot} \leq 
	c\left(\|g\|_{q,\ot}+1\right)^{\alpha_0}\|w\|_{\infty,\ot}^{\frac{r-\alpha}{r}}\|w\|_{\alpha,\ot}^{\frac{\alpha}{r}}
	%\\
	%	&\leq & \frac{1}{2}\|w\|_{\infty,\ot}
	%	+c\left(\|g\|_{q,\ot}+1\right)^{\alpha_0\frac{r}{\alpha}}\|w\|_{\alpha,\ot}
	\end{equation}
	from whence follows
	\begin{equation}
	\|w\|_{\infty,\ot} \leq
	c\left(\|g\|_{q,\ot}+1\right)^{\alpha_0\frac{r}{\alpha}}\|w\|_{\alpha,\ot}.
	\end{equation}
	Recall the definition of $w$ and use \eqref{log25} to obtain
	\begin{eqnarray}
	\|v\|_{\infty ,\ot}\leq\|w\|_{\infty ,\ot}&\leq&
	c\left(\|g\|_{q,\ot}+1\right)^{\alpha_0\frac{r}{\alpha}}(\|v\|_{\alpha,\ot}+|\ot|^{\frac{1}{\alpha}})\nonumber\\
	&\leq&c\left(\|g\|_{q,\ot}+1\right)^{\alpha_0\frac{r}{\alpha}}(\|g\|_{1,\ot}^{\frac{1}{\alpha}}+c)\nonumber\\
	&\leq&c\left(\|g\|_{q,\ot}+1\right)^{\alpha_0\frac{r}{\alpha}+\frac{1}{\alpha}}.
	\end{eqnarray}
	%In view of (1.24), (1.34), (1.22) and (1.19), we obtain
	That is,
	\begin{equation}
	e^{\|u\|_{\infty,\ot}}\leq
	c\left(\|g\|_{q}+1\right)^{\alpha_0\frac{r}{\alpha}+\frac{1}{\alpha}}.
	\end{equation}
	Take the logarithm and make a note of \eqref{key} to get
	\begin{equation}
	\|u\|_{\infty,\ot}\leq c(\ln (\|f\|_{q,\ot}+1)+1).
	\end{equation}
	Finally, we derive \begin{eqnarray*} \|\varphi\|_\infty &\leq
		&\|\varphi_1\|_\infty+\|\varphi_2\|_\infty\\
		&\leq & c\|f\|_{1+\frac{N}{2},\ot}(\ln (\|f\|_{q,\ot}+1)+1)
		+\|\varphi_0\|_\infty.
	\end{eqnarray*}
	This completes the proof.
\end{proof}

\end{document}